\documentclass[11pt]{amsart}

\usepackage{amssymb}
\usepackage{overpic}
\usepackage{enumitem}   
\usepackage{graphicx} 
\usepackage{amsrefs}
\usepackage{xcolor}
\usepackage[hidelinks]{hyperref}

\graphicspath{{Figures/}} 

\newtheorem{remark}{Remark} 
\newtheorem{lemma}{Lemma} 
 
\newtheorem{example}{Example}
\newtheorem{conjecture}{Conjecture} 

\newtheorem{main}{Theorem} 

\begin{document}
	
\title[The hybrid matching of Hurwitz systems]{The hybrid matching of Hurwitz systems}

\author[Luis Fernando Mello and Paulo Santana]{Luis Fernando Mello $^1$ and Paulo Santana$^2$}

\address{$^1$ Instituto de Matem\'{a}tica e Computa\c c\~ao, Universidade Federal de Itajub\'{a}, Avenida BPS 1303, Pinheirinho, CEP 37.500-903, Itajub\'{a}, MG, Brazil}
\email{lfmelo@unifei.edu.br}

\address{$^2$ IBILCE--UNESP, CEP 15054--000, S. J. Rio Preto, S\~ao Paulo, Brazil}
\email{paulo.santana@unesp.br}



\begin{abstract}
	In this paper we study planar hybrid systems composed by two stable linear systems, defined by Hurwitz matrices, in addition with a jump that can be a piecewise linear, a polynomial or an analytic function. We provide an explicit analytic necessary and sufficient condition for this class of hybrid systems to be asymptotically stable. We also prove the existence of limit cycles in this class of hybrid systems. Our results can be seen as generalizations of results already obtained in the literature. This was possible due to an embedding of piecewise smooth vector fields in a hybrid structure.
\end{abstract}

\maketitle

\section{Introduction}

Consider a real autonomous system of differential equations
\begin{equation}\label{1}
	\dot x_1=P_1(x), \quad \dots,  \quad \dot x_n=P_n(x),
\end{equation}
with $x=(x_1,\dots,x_n)\in\mathbb{R}^n$ and $P_1,\dots, P_n\colon\mathbb{R}^n\to\mathbb{R}$ of class $C^1$, where the dot denotes the derivative with respect to the independent variable $t$ (time). As usual, we also identify system \eqref{1} with the vector field $X=(P_1,\dots,P_n)$. 

We say that $x_0\in\mathbb{R}^n$ is a \emph{singularity} of \eqref{1} if $X(x_0)=0$. We say that a singularity $x_0$ is \emph{locally asymptotically stable} if there is a neighborhood $U\subset\mathbb{R}^n$ of $x_0$ such that all orbits of \eqref{1} with initial points in $U$ tend to $x_0$ in forward time. The \emph{basin of attraction} of $x_0$ is the largest $U$ satisfying this condition. If the basin of attraction is the entire $\mathbb{R}^n$, then we say that $x_0$ is \emph{globally asymptotically stable} (GAS). After a translation if necessary, observe that we can assume that $x_0$ is the origin. If the origin is GAS, then by abuse of notation we say that $X$ is GAS.

The ability to determine the basin of attraction of a singularity is of great importance for applications of systems of ordinary differential equations (ODE). For example in \emph{evolutionary games} (e.g. a game that models a conflict between different
species of animals~\cite{SchSig} or the level of corruption of a democratic society~\cites{Acc1,Acc2}), an \emph{evolutionary stable strategy} (ESS) represents a ``uninvadable'' state of the population, in the sense that small deviants behavior will eventually disappear under natural selection~\cite{Fri1991}. Under the language of ODE an ESS is represented by a locally asymptotically stable singularity~\cites{HSS,Palm} and its basin of attraction determines \emph{how deviant} a behavior can be and yet be tamed by natural selection. In particular if an ESS is GAS, then this mean that natural selection will eliminate \emph{any} deviant behavior. For more information about ESS and its relation to ODE, we refer to \cites{ESS1,ESS2,ESS3} and the references therein.

Despite this importance, so far there is very few practical methods to determine if a given singularity is GAS. Among them one that stands out is the \emph{Markus-Yamabe condition}, related with the following conjecture.

\begin{conjecture}[Markus-Yamabe~\cite{MarYam1960}]
	Let $X=(P_1,\dots,P_n)$ be an autonomous $C^1$-vector field on $\mathbb{R}^n$ having a unique singularity at the origin. If the eigenvalues of the Jacobian matrix $DX(x)$ have negative real part for every $x\in\mathbb{R}^n$, then $X$ is GAS.
\end{conjecture}

The conjecture was proved to be true for $n=2$ independently by Fe{\ss}ler~\cite{Feb1995}, Glutsyuk~\cite{Glu1995} and Gutierrez~\cite{Gut1995}. On the other hand, Cima et al~\cite{Cim1997} proved that it is false for $n\geqslant3$.

Although for $n=2$ the Markus-Yamabe condition is a sufficient condition for GAS, it is not necessary. In fact, consider the vector field $X=(P,Q)$ given by
	\[P(x,y)=-x+xy, \quad Q(x,y)=-y.\] 
Note that the origin is the unique singularity and that it is GAS, due to the existence of a global proper Lyapunov function $L\colon\mathbb{R}^2\to\mathbb{R}$ given by
	\[L(x,y)=\ln(1+x^2)+y^2.\]
Nevertheless, the eigenvalues of $DX(x,y)$ are given by $\lambda_1=-1$ and $\lambda_2=-1+y$ and thus $X$ does not satisfy the Markus-Yamabe condition.

To this end we shall say that a planar $C^1$-vector field satisfying the Markus-Yamabe condition is a \emph{Markus-Yamabe} vector field (MY-vector field). In particular observe that if $X$ is a linear MY-vector field, then it is given by 
	\[X(x,y)=A\left(\begin{array}{c} x \\ y \end{array}\right),\]
where $A$ is a $2\times2$ matrix with eigenvalues having negative real part. Since such matrices are known as \emph{Hurwitz matrices}~\cite{Duan}, we say that such a vector field is a \emph{Hurwitz} vector field.

This lack of practical conditions for GAS is also a problem for other types of vector fields, such as the planar \emph{piecewise smooth vector fields}. Briefly, we recall that a planar piecewise smooth vector field is a tuple $\mathcal{X}=(X^+,X^-;\Sigma)$ such that $\Sigma=h^{-1}(\{0\})$, where $h\colon\mathbb{R}^2\to\mathbb{R}$ is a continuous function and $X^\pm$ are planar $C^1$-vector fields defined in a neighborhood of
	\[\Sigma^\pm=\{(x,y)\in\mathbb{R}^2\colon \pm h(x,y)\geqslant0\}.\]
Let $q\in\mathbb{R}^2$. If $q\in\Sigma^\pm\setminus\Sigma$, then the local trajectory of $\mathcal{X}$ at $q$ is given by the local trajectory of $X^\pm$ at $q$. If $q\in\Sigma$, then the local trajectory can be classified as \emph{crossing}, \emph{sliding} or \emph{escaping}. Moreover, $q$ can also be a new type of singularity, such as \emph{pseudo-singularity}, or \emph{tangential singularity}. See~ Guardia et al~\cite{Guardia}. In this paper we only deal with crossing points. For the definition of the other types of trajectories and for more details on piecewise smooth vector fields, we refer to \cite{CarNovTon2024} and the references therein.

We say that $q\in\Sigma$ is a \emph{crossing} point of $\mathcal{X}$ if
\begin{equation}\label{2}
	\left<X^+(q),\nabla h(q)\right>\cdot\left<X^-(q),\nabla h(q)\right>>0,
\end{equation}
where $\left<\cdot,\cdot\right>$ denotes the standard inner product of $\mathbb{R}^2$. Geometrically \eqref{2} means that $X^+(q)$ and $X^-(q)$ are transversal to $\Sigma$ at $q$ and both point towards either $\Sigma^+$ or $\Sigma^-$. In particular, the local trajectories of $X^\pm$ at $q$ agree on orientation and thus we can define the local trajectory of $\mathcal{X}$ at $q$ as the concatenation to the local trajectories of $X^\pm$.

As far as we know the first approach on the characterization of global asymptotic stability for piecewise smooth vector fields is due to Freire et al~\cite{Frei1998}, where they proved that the continuous matching of two Hurwitz vector fields is GAS. More precisely, if $\mathcal{X}=(X^+,X^-;\Sigma)$ is a planar piecewise linear vector field satisfying the following hypotheses:
\begin{enumerate}
	\item[$(H_1)$] $X^\pm$ are Hurwitz vector fields;
	\item[$(H_2)$] $\Sigma$ is a straight line containing the origin;
	\item[$(H_3)$] $X^+|_\Sigma=X^-|_\Sigma$;
\end{enumerate}
then it is GAS. Braga et al~\cite{LFM1} extended this result by proving that if we replace $H_3$ by:
\begin{enumerate}
	\item[$(H_3')$] The points on $\Sigma\setminus\{0\}$ are of crossing type;
\end{enumerate}
then $\mathcal{X}$ is also GAS. However, later it was proved that further generalizations does not maintain the global stability. More precisely, consider the following generalizations for $H_1$ and $H_2$:
\begin{enumerate}
	\item[$(H_1')$] $X^\pm$ are GAS;
	\item[$(H_1'')$] $X^\pm$ are MY-vector fields;
	\item[$(H_2')$] $\Sigma$ is a polygonal line containing the origin.
\end{enumerate}
It follows from \cite{LFM1,LFM2,LFM3} that none of the following conditions
	\[H_1'\cap H_2\cap H_3', \quad H_1''\cap H_2\cap H_3', \quad H_1\cap H_2'\cap H_3',\]
imply the global asymptotically stability for $\mathcal{X}$. We remark that to know whether a piecewise linear system (and in particular piecewise Hurwitz system) is GAS has applicability in control theory. See Iwatani and Hara~\cite{Control}.

In this paper we generalize these results in a different way. We embed the piecewise smooth vector field $\mathcal{X}$ in a \emph{hybrid structure}, defining a hybrid system. Briefly, a \emph{hybrid dynamical system} is a system whose dynamics are governed both by continuous and discrete laws. That is, a system that can both \emph{flow} and \emph{jump}~\cite{LliSan2025}.

Within this structure we are able to unify the hypotheses 
	\[H_1\cap H_2\cap H_3' \quad \text{and} \quad H_1\cap H_2'\cap H_3',\]
understanding precisely when this new hybrid system is GAS. Moreover new bifurcations are now possible, such as limit cycles.

The paper is organized as follows. In Section~\ref{Sec2} we provide a brief introduction to hybrid systems. The statement of the main result is postponed to Section~\ref{Sec3}. In Section~\ref{Sec4} we have some preliminary and technical results used in the proof of the main result, whose proof is given in Section~\ref{Sec5}. Finally in Section~\ref{Sec6} we have a conclusion and some further thoughts. 

\section{Hybrid systems}\label{Sec2}

Before we state what we mean by a hybrid system in this paper, we present one of the usual examples of the field, the \emph{bouncing ball model}.

\begin{example}[Example~$1$ of~\cite{LliSan2025}]
	Consider the vertical motion of a ball dropped (or tossed) from an initial height $h>0$, with initial velocity $v\in\mathbb{R}$ (here negative velocity means downwards, following gravity, while positive velocity means that the ball was tossed upwards). If the ball is under the acceleration of constant gravity $g>0$, then for $h>0$ the state of the ball is governed by the system of differential equations
	\begin{equation}\label{3}
		\dot h=v, \quad \dot v=-g.
	\end{equation}
	Therefore given any initial condition $q=(h_0,v_0)$, $h_0>0$, it follows from \eqref{3} that the ball reaches the ground $h=0$ after a finite amount of time $t_0>0$, with velocity $v(t_0)<0$. See Figure~\ref{Fig1}.
	\begin{figure}[ht]
		\begin{center}
			\begin{overpic}[height=5cm]{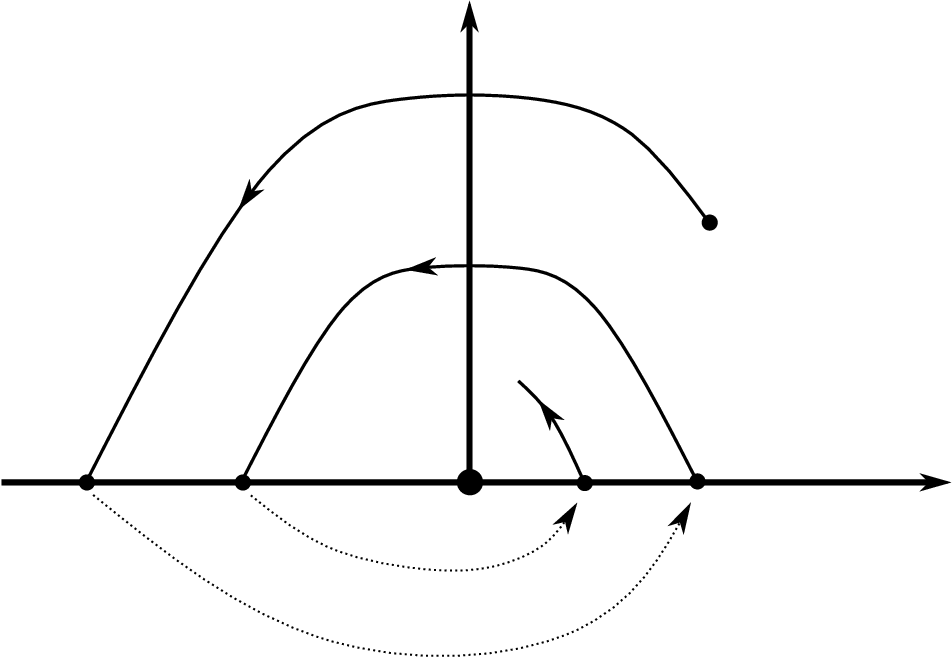} 
				\put(51,65){$h$}
				\put(97,20){$v$}
				\put(45,13){$\mathcal{O}$}
				\put(55,8){$\varphi$}
				\put(65,3){$\varphi$}
				\put(76,45){$q$}			
			\end{overpic}		
		\end{center}
		\caption{Illustration of an orbit of the bouncing ball model. For simplicity, we interchanged the coordinate axes.}\label{Fig1}
	\end{figure}
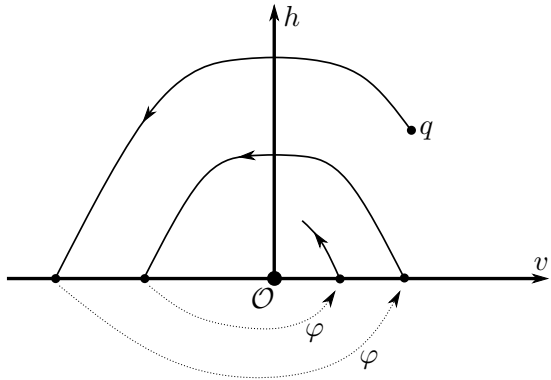
	At this instant, the ball bounces back upwards and its velocity undergoes an instantaneous change $v(t_0)\mapsto -\rho v(t_0)$, modeled by an inelastic collision with dissipate factor $0<\rho<1$. That is, at $h=0$ our system undergoes a jump (or reset) $\varphi$ given by $\varphi(0,v)=(0,-\rho v)$. Roughly speaking, the reset map represents the instantaneous loss of energy that the system suffers when hitting the ground. Observe that any orbit converges to the origin, which represents a ball standing still in the ground.
\end{example}

For more details on the bouncing ball model, we refer to \cite[Section~$2.2.3$]{SchSch2000} and \cite[Section~$1.2$]{Bernardo2008}. For a different example modeling a type of pinball machine with curved surface, we refer to \cite[Example~$2$]{SanSer2025}. For a recent survey in the field, we refer to~\cite{BKPS2023}.

The notion of a hybrid dynamical system is relatively new. Following Schaft and Schumacher~\cite[Section~$1.1$]{SchSch2000}, we quote: 

\emph{[$\dots$] the area of hybrid systems is still in its infancy and a general theory\- of hybrid systems seems premature. More inherently, hybrid systems is such a wide notion that sticking to a single definition shall be too restrictive. [$\dots$] Another difficulty in discussing hybrid systems is that various scientific communities with their own approaches have contributed (and are still contributing) to the area. At least the following three communities can be distinguished. [The] computer science community [$\dots$] [the] modeling and simulation community [$\dots$] [and the] systems and control community}.

Nevertheless, the field is expanding and in particular it has been receiving more attention in recent years from the community of \emph{Qualitative Theory of Differential Equations and Dynamical Systems}. See for example the recent works on Melnikov method~\cites{LiSheZhaHao2016,LiWeiZhaKap2023}, limit cycles~\cite{Lou2024}, topological horseshoes~\cite{WanZha2024}, chaos~\cite{LliSan2025} and polycycles~\cite{SanSer2025}.

One fact that draw attention to the hybrid systems is that a piecewise smooth vector field can be embedded in a hybrid structure, allowing the existence of limit cycles that otherwise would not be possible. For example Llibre and Teixeira~\cite{LliTei2018} proved that a planar piecewise linear vector field given by two centers cannot have limit cycles. However it follows from~\cites{LiLiu2023,LliSan2025} that if we embed such systems in a hybrid structure, then we have the existence of at least one limit cycle.

As anticipated in the Introduction, in this paper we will embedded the the piecewise smooth vector field $\mathcal{X}=(X^+,X^-;\Sigma)$ in a hybrid structure. This embedding will allow us to draw a precise sufficient and necessary condition for global stability and moreover will point the way for the bifurcation of a limit cycle that would not be possible without the hybrid framework.

We now establish what we mean by a planar hybrid system in this paper. Similarly to the definition of a piecewise vector field, a hybrid system is a tuple $\mathfrak{X}=(X^+,X^-;\Sigma;\varphi)$ such that $\Sigma=h^{-1}(\{0\})$, where $h\colon\mathbb{R}^2\to\mathbb{R}$ is a continuous function; $X^\pm$ are planar $C^1$-vector fields defined in a neighborhood of
	\[\Sigma^\pm=\{(x,y)\in\mathbb{R}^2\colon \pm h(x,y)\geqslant0\},\]
and a map $\varphi\colon\Sigma\to\Sigma$ known as \emph{jump} (or reset map).

In this paper $X^\pm$ are Hurwitz vector fields, $\Sigma=\Sigma_{\rho}$ is the zero locus of the function $h_\rho\colon\mathbb{R}^2\to\mathbb{R}$ given by
	\[h_\rho(x,y)=\left\{\begin{array}{ll} 
				y, & \text{if } x\leqslant0, \\  
				y-\rho x, & \text{if } x\geqslant0, 
			\end{array}\right.\] 
where $\rho\in\mathbb{R}_{\geqslant0}$. That is $\Sigma_\rho$ is the polygonal line (also known as broken line) given by $\Sigma_\rho=\Sigma^1\cup\Sigma^2_\rho$, where
\begin{equation}\label{4}
	\Sigma^1=\{(x,y)\in\mathbb{R}^2\colon x\leqslant0,y=0\}, \quad \Sigma^2_\rho=\{(x,y)\in\mathbb{R}^2\colon x\geqslant0,y=\rho x\}.
\end{equation}
We shall also suppose that $\mathfrak{X}$ has the \emph{crossing property}, i.e. we suppose that
\begin{equation}\label{5}
	\left<X^+(q),\nabla h_\rho(q)\right>\cdot\left<X^-(q),\nabla h_\rho(q)\right>>0,
\end{equation}	
for every $q\in\Sigma_\rho\setminus\{(0,0)\}$. In particular we recall that the geometric interpretation of \eqref{5} is that $X^\pm(q)$ are transversal to $\Sigma$ at $q$ and that both vectors $X^+(q)$ and $X^-(q)$ points to the same direction relatively to $\nabla h_\rho(q)$. For hybrid systems without the crossing property and thus allowing other dynamics such as sliding (or \emph{sticking}), we refer to \cite[Section~$6.2.3$]{SchSch2000} and \cite[p.~$82$]{Bernardo2008}.
	
The jump $\varphi=\varphi_\rho\colon\Sigma_\rho\to\Sigma_\rho$ is given by
\begin{equation}\label{6}
	\varphi_\rho(x,y)=\left\{\begin{array}{ll} 
		(-a|x|^r,0), & \text{if } x\leqslant0, \vspace{0.2cm} \\  
		(bx^s,\rho b x^s), & \text{if } x\geqslant0, 
	\end{array}\right.
\end{equation}
with $a$, $b$, $r$, $s\in\mathbb{R}_{>0}$. Observe that $\varphi_\rho$ has inverse given by
\begin{equation}\label{7}
	\varphi_\rho^{-1}(x,y)=\left\{\begin{array}{ll} 
		\bigl(-\bigl|a^{-1}x\bigr|^\frac{1}{r},0\bigr), & \text{if } x\leqslant0, \vspace{0.2cm} \\  
		\bigl(b^{-\frac{1}{s}}x^{\frac{1}{s}},\rho b^{-\frac{1}{s}}x^{\frac{1}{s}}\bigr), & \text{if } x\geqslant0,
	\end{array}\right.
\end{equation} 
for each $\rho\in\mathbb{R}_{\rho\geqslant0}$. Moreover $\Sigma^1$ and $\Sigma^2_\rho$ are invariant by $\varphi_\rho$, i.e. $\varphi_\rho(\Sigma^1)=\Sigma^1$ and $\varphi_\rho(\Sigma^2_\rho)=\Sigma^2_\rho$. See Figure~\ref{Fig2}.
\begin{figure}[ht]
	\begin{center}
		\begin{overpic}[height=5cm]{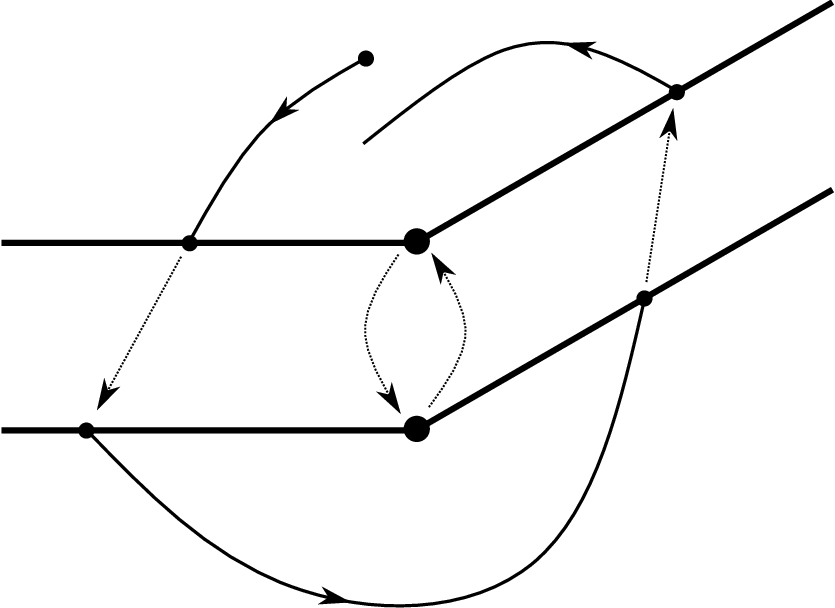} 
			\put(0,45){$\Sigma^1$}
			\put(0,22.5){$\Sigma^1$}
			\put(95,65){$\Sigma^2_\rho$}
			\put(95,42){$\Sigma^2_\rho$}
			\put(47,46){$\mathcal{O}$}
			\put(50,16){$\mathcal{O}$}	
			\put(18,32){$\varphi_\rho$}
			\put(37,32){$\varphi_\rho$}
			\put(56,35){$\varphi_\rho$}
			\put(80,48){$\varphi_\rho$}
			\put(46,65){$q_0$}
			\put(23,40){$q_1$}
			\put(6,17){$q_2$}
			\put(22,54){$X^+$}
			\put(71,10){$X^-$}
		\end{overpic}		
	\end{center}
	\caption{Illustration of $\varphi_\rho$ and $\mathfrak{X}$. The origin $\mathcal{O}$ is interpreted as a singularity. By abuse of notation we drew two copies of $\Sigma_\rho$.}\label{Fig2}
\end{figure}
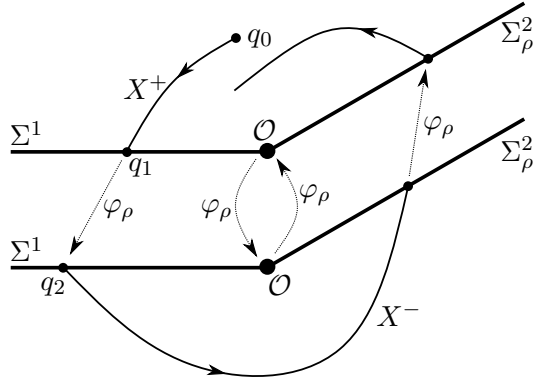

The dynamics of $\mathfrak{X}=(X^+,X^-;\Sigma;\varphi)$ work as follows. Given $q_0\in\mathbb{R}^2\setminus\Sigma$, the local trajectory of $\mathfrak{X}$ at $q_0$ is given by the local trajectory of $X^\pm$. If this trajectory never intersects $\Sigma_\rho$, then we are done. If it does intersect $\Sigma$ at a point $q_1$, then we apply the jump obtaining a point $q_2=\varphi_\rho(q_1)$. It follows from \eqref{5} that the local trajectories of $X^\pm$ at $q_2$ are transversal to $\Sigma$ and agree on orientation. Hence we are able to follow exactly one of such trajectories leaving $\Sigma$. The process now repeats. See Figure~\ref{Fig2}.

In simple words, the dynamic of the hybrid system $\mathfrak{X}=(X^+,X^-;\Sigma_\rho;\varphi)$ is similar to the dynamics of the piecewise vector field $\mathcal{X}=(X^+,X^-;\Sigma_\rho)$, with the exception that when hitting the switching set $\Sigma$ we apply the jump before crossing it. In particular observe from \eqref{5} that if $a=b=r=s=1$, then $\varphi_\rho$ reduces to the identity map and thus the hybrid system becomes a piecewise vector field. 

Similarly to the fact that switching between $X^+$ and $X^-$ has the physic interpretation of a particle switching from one medium to another (e.g. the refraction of light when passing from air to water), the jump $\varphi_\rho$ represents the instantaneous gain or loss of energy that such a particle may suffer when transitioning between the mediums.

\section{Statement of the main result}\label{Sec3}

We now provide the definitions for a precise statement of our main result. Let $\mathfrak{X}=(X^+,X^-;\Sigma_\rho;\varphi_\rho)$ be a hybrid system satisfying the following hypotheses:
\begin{enumerate}
	\item[$(H_1)$] $X^\pm$ are Hurwitz vector fields;
	\item[$(H_2')$] $\Sigma_\rho=\Sigma^1\cup\Sigma^2_\rho$ is the polygonal straight line given by \eqref{4};
	\item[$(H_3')$] $\mathfrak{X}$ has the crossing property \eqref{5}.
\end{enumerate}

It follows from $H_1$ that $X^\pm$ can only have the following types of singulari\-ties at the origin:
\begin{enumerate}[label=(\roman*)]
	\item An attracting node with distinct eigenvalues, denoted by $N_1$;
	\item An attracting non-diagonalizable node, denoted by $N_2$;
	\item An attracting focus, denoted by $F$.
\end{enumerate}
In particular, note that if $X^\pm$ have an attracting star node (i.e. a diagonali\-zable node with equal eigenvalues), then it does satisfy $H_1$. In our main result we provide a complete characterization of the dynamics of $\mathfrak{X}$.

\begin{main}\label{Main1}
	Let $\mathfrak{X}=(X^+,X^-;\Sigma_\rho;\varphi_\rho)$ be a hybrid system satisfying hypotheses $H_1$, $H_2'$ and $H_3'$. Then the following statements hold.
	\begin{enumerate}[label=(\alph*)]
		\item If $X^+$ or $X^-$ is not of type $F$, then $\mathfrak{X}$ is GAS.
		\item If $X^+$ and $X^-$ are both of type $F$, then exactly one of the following statements hold.
		\begin{enumerate}[label=(\roman*)]
			\item The origin is GAS.
			\item The origin is globally asymptotically unstable.
			\item The origin is a global center.
			\item $\mathfrak{X}$ has a unique limit cycle which is hyperbolic.
		\end{enumerate}
	\end{enumerate}
	Moreover, all statements occur and there is an explicit analytic charac\-terization of it.
\end{main}

\section{Preliminary results}\label{Sec4}

In this section we provide some technical results on global normal forms for $\mathfrak{X}$ and on the dynamics of Hurwitz linear vector fields of type $N_1$, $N_2$ and $F$.

\subsection{Global normal form}

\begin{lemma}\label{L1}
	Let $\mathfrak{X}=(X^+,X^-;\Sigma_\rho;\varphi_\rho)$ be a hybrid system satisfying hypotheses $H_1$, $H_2'$ and $H_3'$. Let also $B^\pm$ be the Hurwitz matrices such that
		\[X^\pm(x,y)=B^\pm \left(\begin{array}{c} x \\ y \end{array}\right).\]
	Then there is a topological equivalence between the piecewise linear systems $(X^+,X^-;\Sigma_\rho)$ and $(Y^+,Y^-;\Sigma_\rho)$, where $Y^\pm$ are given by
	\begin{equation*}
		Y^\pm(x,y)=A^\pm \left(\begin{array}{c} x \\ y \end{array}\right), \quad A=\left(\begin{array}{cc} \sigma^\pm & \delta^\pm \\ 1 & 0 \end{array}\right),
	\end{equation*}
	with $\sigma^\pm=\operatorname{tr} B^\pm<0$ and $\delta^\pm=-\det B^\pm<0$.
\end{lemma}

\begin{proof}
	If we let
		\[B^\pm=\left(\begin{array}{cc} b_{11}^\pm & b_{12}^\pm \vspace{0.2cm} \\ b_{21}^\pm & b_{22}^\pm \end{array}\right),\]
	then the crossing property~\eqref{5} restricted at $\Sigma^2_\rho$ writes
		\[\bigl(b_{21}^++(b_{22}^+-b_{11}^+)\rho-b_{12}^+\rho^2\bigr)\bigl(b_{21}^-+(b_{22}^--b_{11}^-)\rho-b_{12}^-\rho^2\bigr)x^2>0.\]
	In particular, we obtain
	\begin{equation}\label{9}
		\eta^\pm:=b_{21}^\pm+(b_{22}^\pm-b_{11}^\pm)\rho-b_{12}^\pm\rho^2\neq0.
	\end{equation}
	Consider the matrices 
		\[C^\pm=\left(\begin{array}{cc} c_{11}^\pm & c_{12}^\pm \vspace{0.2cm} \\ c_{21}^\pm & c_{22}^\pm \end{array}\right),\]
	where
	\[\begin{array}{cc}
		\displaystyle c_{11}^\pm=1-\frac{(\delta^\pm-b_{12}^\pm)\rho+b_{22}}{\eta^\pm}\rho, &\displaystyle c_{12}^\pm=\frac{(\delta^\pm-b_{12}^\pm)\rho+b_{22}}{\eta^\pm}, \vspace{0.2cm} \\
		\displaystyle c_{21}^\pm=\rho+\frac{b_{12}^\pm\rho^2+b_{11}^\pm\rho-1}{\eta^\pm}\rho, &\displaystyle c_{22}^\pm=-\frac{b_{12}^\pm\rho^2+b_{11}^\pm\rho-1}{\eta^\pm}.
	\end{array}\]
	It follows from \eqref{9} that $C^\pm$ are well defined. Moreover we observe that
		\[\det C^\pm=-\frac{\delta^\pm\rho^2+\sigma^\pm\rho-1}{\eta^\pm}.\]
	Hence it follows from \eqref{9}, $\delta^\pm<0$, $\sigma^\pm<0$ and $\rho\geqslant0$ that $\det C^\pm\neq0$ and thus $C^\pm$ is a linear change of variables. From straightforward calculations one can see that
		\[A^\pm=C^\pm B^\pm(C^\pm)^{-1} \quad \text{and} \quad C^\pm\left(\begin{array}{c} x \\ y \end{array}\right)=\left(\begin{array}{c} x \\ y \end{array}\right), \quad \text{if} \quad (x,y)\in\Sigma_\rho.\]
	Therefore if we define the map $H\colon\mathbb{R}^2\to\mathbb{R}^2$ by
		\[H(x,y)=\left\{\begin{array}{ll}
					C^+\left(\begin{array}{c} x \\ y \end{array}\right), & \text{if } (x,y)\in\Sigma^+, \vspace{0.2cm} \\
					C^-\left(\begin{array}{c} x \\ y \end{array}\right), & \text{if } (x,y)\in\Sigma^-,
				\end{array}\right.\]
	then it is a well defined topological equivalence between the piecewise systems $(X^+,X^-;\Sigma_\rho)$ and $(Y^+,Y^-;\Sigma_\rho)$. 
\end{proof}

For simplicity throughout the remaining of this section we state our results only for $A^+$. Similar results can be proved for $A^-$.

\subsection{Hurwitz vector fields of type \boldmath{$N_1$}}

Suppose that the origin is an attracting node of $X^+$ with distinct eigenvalues given by
	\[r_1^+=\frac{\sigma^+}{2}+\frac{\sqrt{(\sigma^+)^2+4\delta^+}}{2}, \quad r_2^+=\frac{\sigma^+}{2}-\frac{\sqrt{(\sigma^+)^2+4\delta^+}}{2}.\]
In particular, observe that $r_2^+<r_1^+<0$. Following~\cite{LFM1} we define the sets
	\[\begin{array}{l}
		R_1^+=\{(x,y)\in\Sigma_\rho^+\setminus\Sigma_\rho\colon x>r_1^+y\}, \vspace{0.2cm} \\ E_1^+=\{(x,y)\in\Sigma_\rho^+\setminus\Sigma_\rho\colon x=r_1^+y\}, \vspace{0.2cm} \\
		R_2^+=\{(x,y)\in\Sigma_\rho^+\setminus\Sigma_\rho\colon r_2^+y<x<r_1^+y\}, \vspace{0.2cm} \\
		E_2^+=\{(x,y)\in\Sigma_\rho^+\setminus\Sigma_\rho\colon x=r_2^+y\}, \vspace{0.2cm} \\
		R_3^+=\{(x,y)\in\Sigma_\rho^+\setminus\Sigma_\rho\colon x<r_2^+y\},
	\end{array}\]
see Figure~\ref{Fig3}. 
\begin{figure}[ht]
	\begin{center}
		\begin{overpic}[width=8cm]{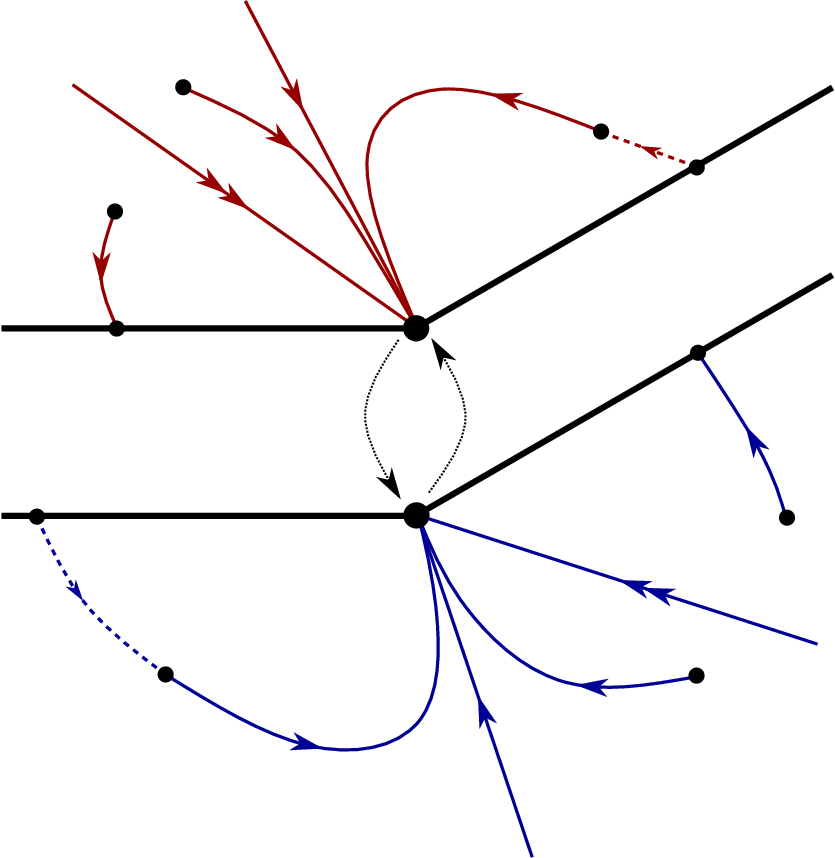} 
			\put(75,95){$R_1^+$}
			\put(15,95){$R_2^+$}
			\put(0,75){$R_3^+$}
			\put(30.5,98){$E_1^+$}
			\put(2.5,90){$E_2^+$}
			\put(68.5,81){$q$}
			\put(19,86){$q$}
			\put(15,75){$q$}
			\put(7,56){$X^+(t^2_q,q)$}
			\put(78,75){$X^+(t^1_q,q)$}
			\put(25,42){$\Sigma_\rho$}
			\put(95,91){$\Sigma_\rho$}
			\put(46,56){$\mathcal{O}$}
			\put(44,35){$\mathcal{O}$}
			\put(55.5,51){$\varphi_\rho$}
			\put(36.5,51){$\varphi_\rho$}
			\put(10,10){$R_1^-$}
			\put(80,10){$R_2^-$}
			\put(75,40){$R_3^-$}
			\put(63,-1){$E_1^-$}
			\put(95,23){$E_2^-$}
		\end{overpic}		
	\end{center}
	\caption{Illustration of two Hurwitz vector fields $X^\pm$ of type $N_1$. The curves in red (resp. blue) are the solutions of $X^+$ (resp. $X^-$). For interpretation of the references to color in this figure legend, the reader is referred to the web version of this article.}\label{Fig3}
\end{figure}
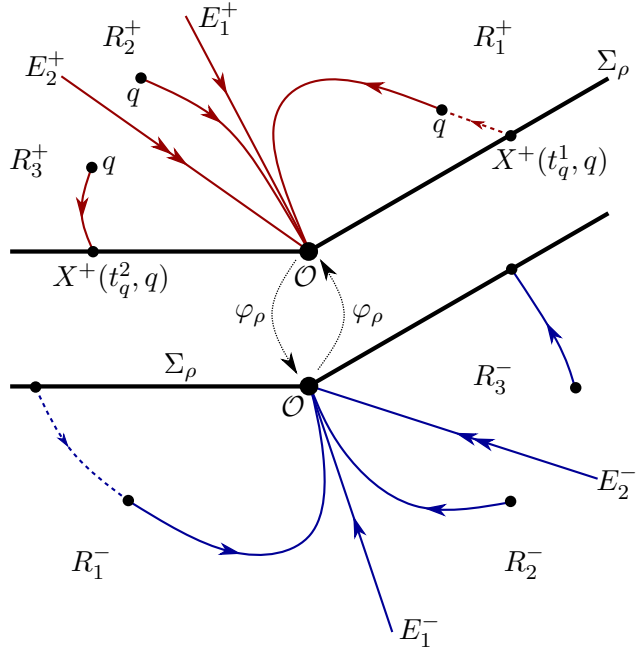

Given $q\in\Sigma_\rho^+$, let $X^+(t,q)$ denote the solution of $X^+$ with initial condition $X^+(0,q)=0$. The following lemma follows similarly to~\cite[Lemma~$5$]{LFM1}.

\begin{lemma}\label{L2}
	Let $X^+$ be a Hurwitz vector field of type $N_1$ defined in $\Sigma_\rho^+$. Then the following statements hold (see Figure~\ref{Fig3}).
	\begin{enumerate}[label=(\alph*)]
		\item If $q\in E_1^+\cup E_2^+$, then $X^+(t,q)\to\mathcal{O}$ as $t\to+\infty$.
		\item If $q\in R_1^+$, then
		\begin{enumerate}[label=(\roman*)]
			\item $X^+(t,q)\to\mathcal{O}$ as $t\to+\infty$;
			\item there is $t^1_q<0$ such that $X^+(t^1_q,q)\in\Sigma_\rho$.
		\end{enumerate}
		\item If $q\in R_2^+$, then $X^+(t,q)\to\mathcal{O}$ as $t\to+\infty$.
		\item If $q\in R_3^+$, then there is $t^2_q>0$ such that $X^+(t^2_q,q)\in\Sigma_\rho$.
	\end{enumerate}
\end{lemma}

\subsection{Hurwitz vector fields of type \boldmath{$N_2$}}

Suppose that the origin is an attracting non-diagonalizable node of $X^+$ with repeated eigenvalue given by
	\[r^+=\frac{\sigma^+}{2}<0.\]
Following~\cite{LFM1} we define the sets
	\[\begin{array}{l}
		S_1^+=\{(x,y)\in\Sigma_\rho^+\setminus\Sigma_\rho\colon x>r^+y\}, \vspace{0.2cm} \\ 	
		E^+=\{(x,y)\in\Sigma_\rho^+\setminus\Sigma_\rho\colon x=r^+y\}, \vspace{0.2cm} \\
		S_2^+=\{(x,y)\in\Sigma_\rho^+\setminus\Sigma_\rho\colon x<r^+y\},
\end{array}\]
see Figure~\ref{Fig4}. 
\begin{figure}[ht]
	\begin{center}
		\begin{overpic}[width=8cm]{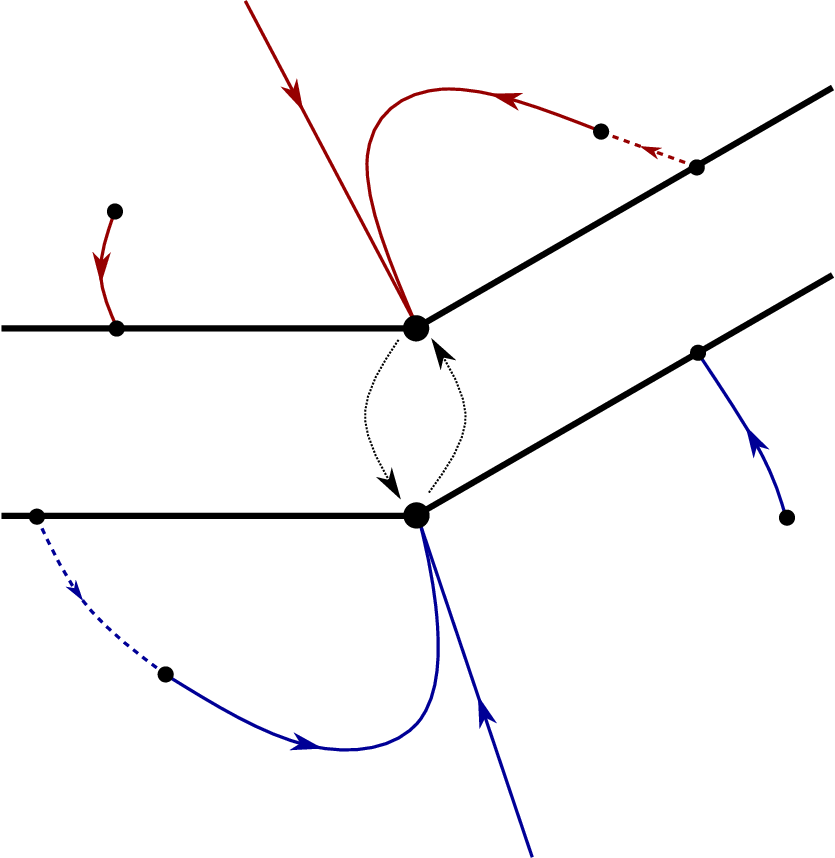} 
			\put(75,95){$S_1^+$}
			\put(10,85){$S_2^+$}
			\put(30.5,98){$E^+$}
			\put(68.5,81){$q$}
			\put(15,75){$q$}
			\put(7,56){$X^+(t^2_q,q)$}
			\put(78,75){$X^+(t^1_q,q)$}
			\put(25,42){$\Sigma_\rho$}
			\put(95,91){$\Sigma_\rho$}
			\put(46,56){$\mathcal{O}$}
			\put(44,35){$\mathcal{O}$}
			\put(55.5,51){$\varphi_\rho$}
			\put(36.5,51){$\varphi_\rho$}
			\put(10,10){$S_1^-$}
			\put(75,25){$S_2^-$}
			\put(63,-1){$E^-$}
		\end{overpic}		
	\end{center}
	\caption{Illustration of two Hurwitz vector fields $X^\pm$ of type $N_2$. The curves in red (resp. blue) are the solutions of $X^+$ (resp. $X^-$). For interpretation of the references to color in this figure legend, the reader is referred to the web version of this article.}\label{Fig4}
\end{figure}
The following lemma follows from~\cite[Lemma~$6$]{LFM1}.

\begin{lemma}\label{L3}
	Let $X^+$ be a Hurwitz vector field of type $N_2$ defined in $\Sigma_\rho^+$. Then the following statements hold (see Figure~\ref{Fig4}).
	\begin{enumerate}[label=(\alph*)]
		\item If $q\in E^+$, then $X^+(t,q)\to\mathcal{O}$ as $t\to+\infty$.
		\item If $q\in S_1^+$, then
		\begin{enumerate}[label=(\roman*)]
			\item $X^+(t,q)\to\mathcal{O}$ as $t\to+\infty$;
			\item there is $t^1_q<0$ such that $X^+(t^1_q,q)\in\Sigma_\rho$.
		\end{enumerate}
		\item If $q\in S_2^+$, then there is $t^2_q>0$ such that $X^+(t^2_q,q)\in\Sigma_\rho$.
	\end{enumerate}
\end{lemma}

\subsection{Hurwitz vector fields of type \boldmath{$F$}}

Suppose that the origin is an attracting focus of $X^+$ with eigenvalues given by
	\[r_1^+=\lambda^++i\mu_+, \quad r_2^+=\lambda^+-i\mu_+,\]
where
	\[\lambda^+=\frac{\sigma^+}{2}<0, \quad \mu^+=\frac{\sqrt{|(\sigma^+)^2+4\delta^+|}}{2}>0,\]
and $i$ is the imaginary unit satisfying $i^2=-1$. The following lemma follows from~\cite[Lemma~$7$]{LFM1}.

\begin{lemma}\label{L4}
	Let $X^+$ be a Hurwitz vector field of type $F$ defined in $\Sigma_\rho^+$. Then for every $q\in\Sigma_\rho^+\setminus\Sigma_\rho$ there are $t^1_q<0<t^2_q$ such that $X^+(t^1_q,q)\in\Sigma_\rho$ and $X^+(t^2_q,q)\in\Sigma_\rho$ (see Figure~\ref{Fig5}).
\end{lemma}

\begin{figure}[ht]
	\begin{center}
		\begin{overpic}[width=8cm]{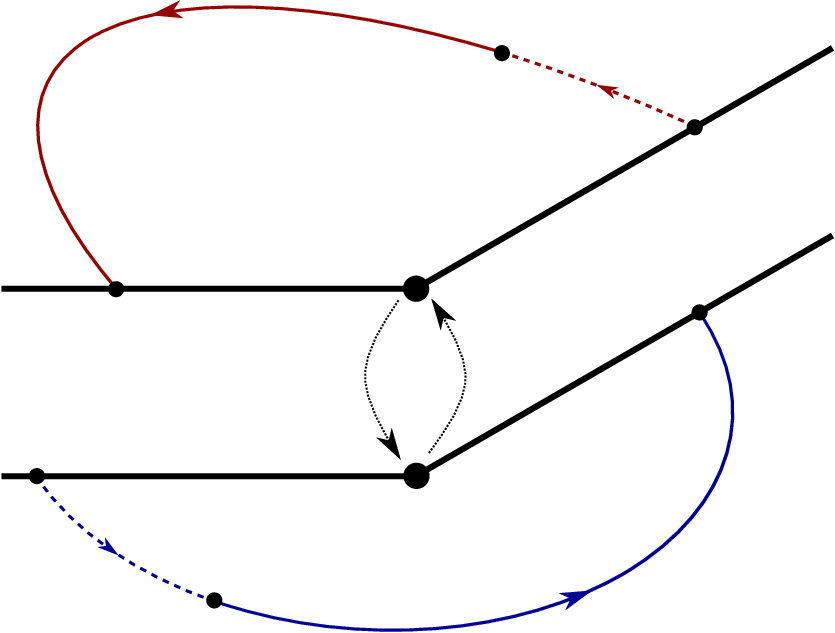} 
			\put(58,65.5){$q$}
			\put(7,35){$X^+(t^2_q,q)$}
			\put(79,54){$X^+(t^1_q,q)$}
			\put(25,21){$\Sigma_\rho$}
			\put(97,71){$\Sigma_\rho$}
			\put(50,14){$\mathcal{O}$}
			\put(47,35){$\mathcal{O}$}
			\put(57,30){$\varphi_\rho$}
			\put(37.5,30){$\varphi_\rho$}
		\end{overpic}		
	\end{center}
	\caption{Illustration of two Hurwitz vector fields $X^\pm$ of type $F$. The curves in red (resp. blue) are the solutions of $X^+$ (resp. $X^-$). For interpretation of the references to color in this figure legend, the reader is referred to the web version of this article.}\label{Fig5}
\end{figure}
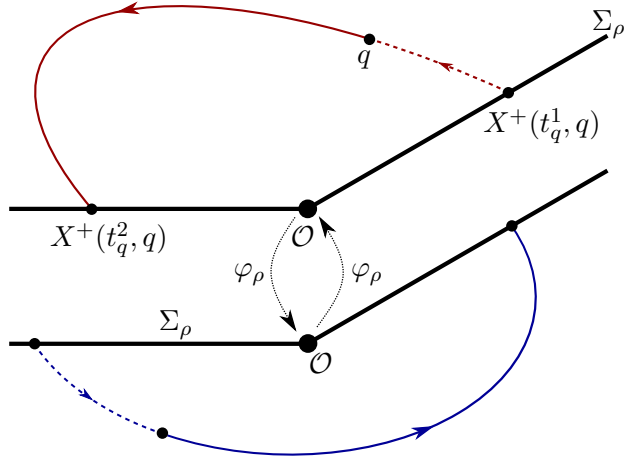

To end this section, we observe that since $X^+$ is linear we can compute explicitly its solution, see~\cite[Chapter~$1$]{Perko}. Therefore it is not hard to see that if $X^+$ is a Hurwitz vector field of type $F$ with the normal form given by Lemma~\ref{L1}, then its solution 
	\[X^+(t;x,y)=\bigl(x^+(t;x,y),y^+(t;x,y)\bigr)\]
is given by
\begin{equation}\label{10}
	\begin{array}{l}
		\displaystyle x^+(t;x,y)=e^{\lambda^+t}\left(x\cos(\mu^+t)+\frac{1}{\mu^+}(\lambda^+x-r_1^+r_2^+y)\sin(\mu^+t)\right), \vspace{0.2cm} \\
		\displaystyle y^+(t;x,y)=e^{\lambda^+t}\left(y\cos(\mu^+t)+\frac{1}{\mu^+}(x-\lambda^+y)\sin(\mu^+t)\right).
	\end{array}
\end{equation}

\section{Proof of Theorem~\ref{Main1}}\label{Sec5}

\begin{proof}[Proof of Theorem~\ref{Main1}]
	It follows from Lemmas~\ref{L2}, \ref{L3} and \ref{L4} that if $X^+$ or $X^-$ is not of type $F$, then $\mathfrak{X}$ is GAS. More precisely, if we suppose for example that $X^+$ is of type $N_1$ and $X^-$ is of type $N_2$ (i.e. $N_1-N_2$), then it is clear that every orbit of $\mathfrak{X}$ will eventually go to the origin. See Figures~\ref{Fig3} (in red) and~\ref{Fig4} (in blue). The cases $N_1-N_1$, $N_1-F$, $N_2-N_2$, $N_2-F$ and their symmetric (e.g. $F-N_1$) follow similarly. This proves statement $(a)$.
	
	We now focus on statement $(b)$. In particular, from now on we suppose that both $X^\pm$ are of type $F$. Consider the displacement map $\Delta\colon\mathbb{R}_{>0}\to\mathbb{R}$ given by
	\begin{equation}\label{11}
		\Delta(x)=x^-\bigl(-t^-;\varphi^{-1}_\rho(x,\rho x)\bigr)-\varphi_\rho\bigl(x^+(t^+;x,\rho x)\bigr),
	\end{equation}
	where $t^\pm>0$ are such that
	\begin{equation}\label{12}
		y^+(t^+;x,\rho x)=0, \quad y^-\bigl(-t^-;\varphi^{-1}_\rho(x,\rho x)\bigr)=0.
	\end{equation}
	See Figure~\ref{Fig6}.
	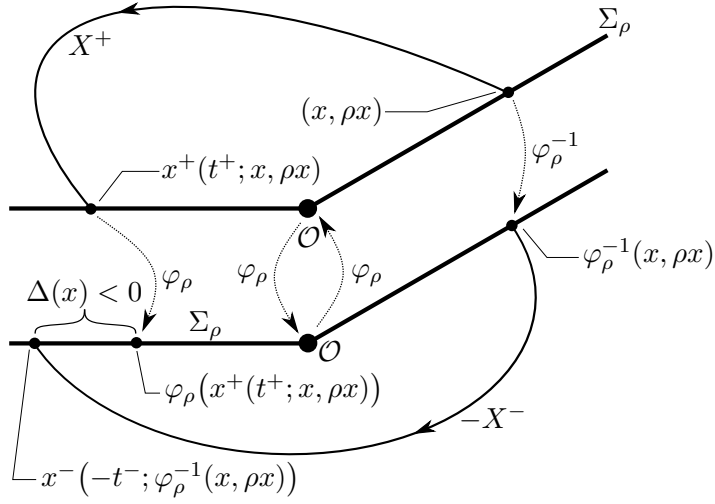
\begin{figure}[ht]
		\begin{center}
			\begin{overpic}[width=8cm]{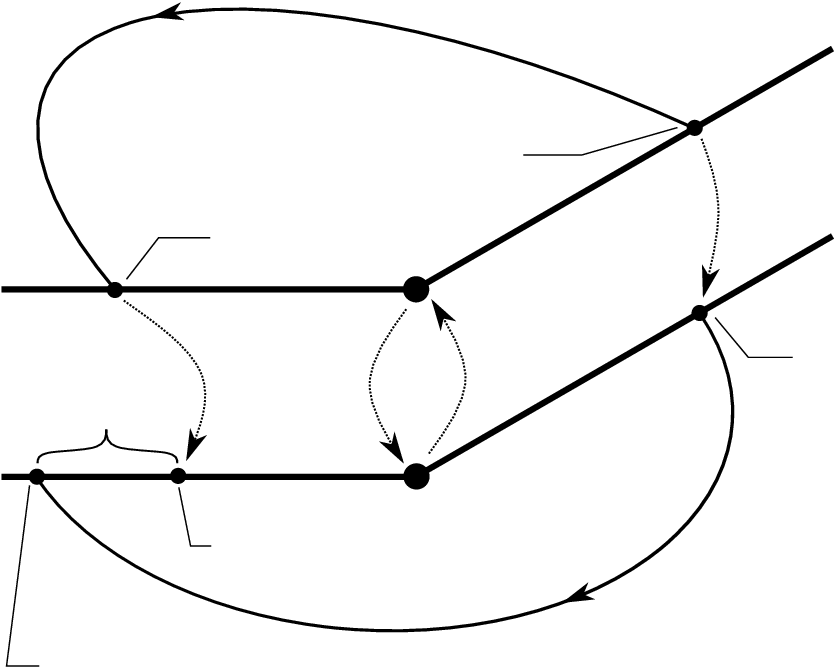} 
				\put(30,25){$\Sigma_\rho$}
				\put(98,75.5){$\Sigma_\rho$}
				\put(51.5,20){$\mathcal{O}$}
				\put(48,39){$\mathcal{O}$}
				\put(57,34){$\varphi_\rho$}
				\put(38,34){$\varphi_\rho$}
				\put(25.5,33){$\varphi_\rho$}
				\put(87,54){$\varphi^{-1}_\rho$}
				\put(48.5,60.25){$(x,\rho x)$}
				\put(95.5,36.5){$\varphi_\rho^{-1}(x,\rho x)$}
				\put(5.5,-0.8){$x^-\bigl(-t^-;\varphi^{-1}_\rho(x,\rho x)\bigr)$}
				\put(25.5,50.5){$x^+(t^+;x,\rho x)$}
				\put(26,14){$\varphi_\rho\bigl(x^+(t^+;x,\rho x)\bigr)$}
				\put(3,30){$\Delta(x)<0$}
				\put(10,71){$X^+$}
				\put(75,9){$-X^-$}
			\end{overpic}		
		\end{center}
		\caption{Illustration of the displacement map $\Delta$.}\label{Fig6}
	\end{figure}

	From \eqref{6} and \eqref{10} we have
	\begin{equation}\label{13}
		\varphi_\rho\bigl(x^+(t^+;x,\rho x)\bigr)=-ae^{\lambda^+t^+r}|\cos(\mu^+t^+)+\Phi^+\sin(\mu^+t^+)|^rx^r,		
	\end{equation}
	where $\Phi^\pm=(\lambda^\pm-r_1^\pm r_2^\pm\rho)/\mu^\pm$. On the other hand, from \eqref{7} and \eqref{10} we obtain
	\begin{equation}\label{14}
		x^-\bigl(-t^-;\varphi^{-1}_\rho(x,\rho x)\bigr)=e^{-\lambda^-t^-}\bigl(\cos(\mu^-t^-)-\Phi^-\sin(\mu^-t^-)\bigr)b^{-\frac{1}{s}}x^\frac{1}{s}.
	\end{equation}
	Therefore it follows from \eqref{11}, \eqref{13} and \eqref{14} that $\Delta(x)<0$ if and only if
	\begin{equation}\label{15}
		ab^\frac{1}{s}x^r<e^{|\lambda^-t^-+\lambda^+t^+r|}\frac{|\cos(\mu^-t^-)-\Phi^-\sin(\mu^-t^-)|}{|\cos(\mu^+t^+)+\Phi^+\sin(\mu^+t^+)|^r}x^\frac{1}{s}.
	\end{equation}
	From \eqref{15} we can deduce statements $(i)$, $(ii)$, $(iii)$ and $(iv)$. More precisely, if $r=1/s$ then we can factor $x$ out of \eqref{15} and thus obtain that the origin is globally asymptotically stable (resp. globally asymptotically unstable) if \eqref{15} is satisfied (resp. is the opposite inequality is satisfied). If the equality holds, then the origin is a global center.
	
	If $r\neq1/s$ then we can divide \eqref{15} either by $x^r$ or $x^\frac{1}{s}$ (whichever is smaller) and thus obtain that there is exactly one $x_0>0$ satisfying the equality in~\eqref{15}. This $x_0$ represents a limit cycle. Differentiating $\Delta(x)$ in this case one can see that $\Delta'(x_0)\neq0$ and thus the limit cycle is hyperbolic. This finishes the proof of statement $(b)$.
\end{proof}

\begin{remark}[The case $\rho=0$]
	Solving \eqref{12} we obtain,
	\begin{equation}\label{16}
		t^\pm(\rho)=\frac{1}{\mu^\pm}\left(\pm\arctan\left(\frac{\mu^\pm}{\lambda^\pm\rho-1}\rho\right)+\pi\right).
	\end{equation}
	In particular we have $t^\pm(0)=\pi/\mu^\pm$ and thus if we replace $\rho=0$ at \eqref{15} we obtain,
		\[ab^\frac{1}{s}x^r<e^{\bigl|\frac{\lambda^-}{\mu^-}+\frac{\lambda^+}{\mu^+}r\bigr|\pi}x^\frac{1}{s}.\]
	Observe that if $a=b=r=s=1$, then the above inequality is satisfied regardless of $\lambda^\pm$ and $\mu^\pm$, i.e. regardless of $X^\pm$. This in addition with Theorem~\ref{Main1}$(a)$ ensures that if the jump $\varphi_\rho$ is reduced to the identity map (i.e. if the hybrid system $\mathfrak{X}$ is reduced to a piecewise linear vector field) and $\Sigma_\rho$ is reduced to a straight line, then $\mathfrak{X}$ is GAS. This is precisely the main result of \cite{LFM1}.
\end{remark}

\begin{remark}[The limit case $\rho\to+\infty$]
	From \eqref{16} we have,	
		\[\lim\limits_{\rho\to+\infty}t^\pm(\rho)=\frac{1}{\mu^\pm}\left(\pm\arctan\left(\frac{\mu^\pm}{\lambda^\pm}\right)+\pi\right).\]
	This in addition with the following well-known properties of the trigonometric functions
	\[\begin{array}{ll}
		\displaystyle \cos(x+\pi)=-\cos x, &\displaystyle \sin(x+\pi)=-\sin(x), \vspace{0.2cm} \\
		\displaystyle \cos(\arctan x)=\frac{1}{\sqrt{1+x^2}}, &\displaystyle \sin(\arctan x)=\frac{x}{\sqrt{1+x^2}},
	\end{array}\]
	and some straightforward calculations, allow us to conclude that
	\begin{equation}\label{17}
		\lim\limits_{\rho\to+\infty}\frac{|\cos(\mu^-t^-)-\Phi^-\sin(\mu^-t^-)|}{|\cos(\mu^+t^+)+\Phi^+\sin(\mu^+t^+)|^r}=\left\{\begin{array}{ll} 0, & \text{if } r>1, \vspace{0.2cm} \\ +\infty, &\text{if } r<1, \vspace{0.2cm} \\ \displaystyle \sqrt{\frac{(\lambda^-)^2+(\mu^-)^2}{(\lambda^+)^2+(\mu^+)^2}}, &\text{if } r=1. \end{array}\right.
	\end{equation}
	Since the above limit is the only factor of \eqref{15} that depends on $\rho$, we conclude that if $r\neq1$ then the right-hand side of \eqref{15} can either explode or collapse to zero as $\rho\to+\infty$. In particular, if we assume that $1/s=r$ (and thus factoring $x$ out of \eqref{15}), then it follows from \eqref{17} that the stability of the origin can reverse as $\rho\to+\infty$. 
\end{remark}

We finish this section with an example that shows that even if $\Sigma_\rho$ a straight line (i.e. $\rho=0$) and $X^\pm$ has a continuous match at $\Sigma$, the hybrid system $\mathfrak{X}$ can have a limit cycle. We observe that the existence of such a limit cycle was proved impossible in the previous piecewise framework~\cites{LFM1,Frei1998}.

\begin{example}
	Consider the hybrid system $\mathfrak{X}=(X^+,X^-;\Sigma;\varphi)$ obtained by replacing 
	\begin{equation}\label{18}
		a=b=1, \quad r=s=3, \quad \rho=0, \quad \lambda^\pm=-1, \quad \mu^\pm=1,
	\end{equation}
	in the proof of Theorem~\ref{Main1}. Using the global normal form provided by Lemma~\ref{L1}, we have that $\mathfrak{X}$ is given by
		\[\begin{array}{l}
			X^\pm(x,y)=(-2x-2y,x), \vspace{0.2cm} \\
			\Sigma=\{(x,y)\in\mathbb{R}^2\colon y=0\}, \vspace{0.2cm} \\
			\varphi(x,0)=(x^3,0).
		\end{array}\]
	Replacing \eqref{18} into \eqref{13}, \eqref{14} and \eqref{16} we obtain
	\begin{equation}\label{19}
		\Delta(x)=-e^\pi x^\frac{1}{3}+e^{-3\pi}x^3,
	\end{equation}
	where we recall that $\Delta\colon\mathbb{R}_{>0}\to\mathbb{R}$ is the displacement map given by~\eqref{11}. Solving \eqref{19} we obtain that $x_0=e^{\frac{3}{2}\pi}$ is the unique solution of $\Delta(x)=0$. Hence, $\mathfrak{X}$ has a unique limit cycle $\gamma$. Differentiating \eqref{19} we obtain
		\[\Delta'(x)=-\frac{1}{3}e^\pi x^{-\frac{2}{3}}+3e^{-3\pi}x^2.\]
	Since $\Delta'(x_0)=8/3$, we have that $\gamma$ is hyperbolic and unstable. In particular, from the uniqueness of $\gamma$ we have that the origin is locally asymptotically stable. See Figure~\ref{Fig7}.
	\begin{figure}[ht]
		\begin{center}
			\begin{overpic}[width=8cm]{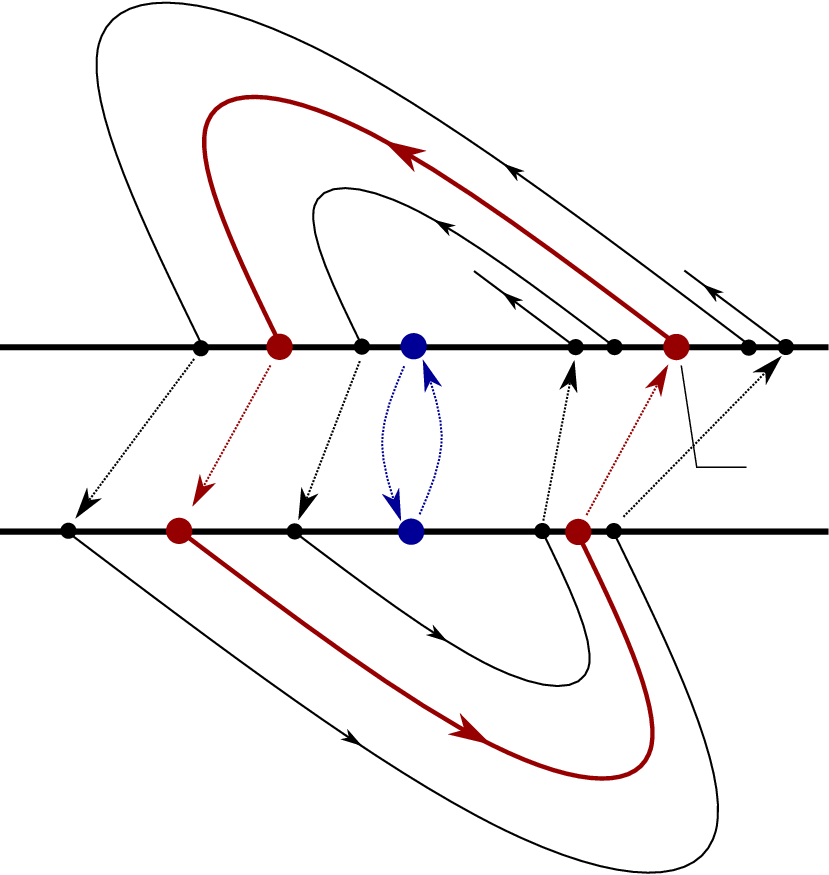} 
				\put(45,62.5){$\mathcal{O}$}
				\put(45,34){$\mathcal{O}$}
				\put(60,9){$\gamma$}
				\put(25,90.5){$\gamma$}
				\put(0,40){$\Sigma$}
				\put(0,61){$\Sigma$}
				\put(28,50){$\varphi$}
				\put(68.5,50){$\varphi$}
				\put(86,45.5){$(x_0,0)$}
			\end{overpic}		
		\end{center}
		\caption{Illustration of $\mathfrak{X}$. Here blue denotes stable and red denotes unstable. For interpretation of the references to color in this figure legend, the reader is referred to the web version of this article.}\label{Fig7}
	\end{figure}
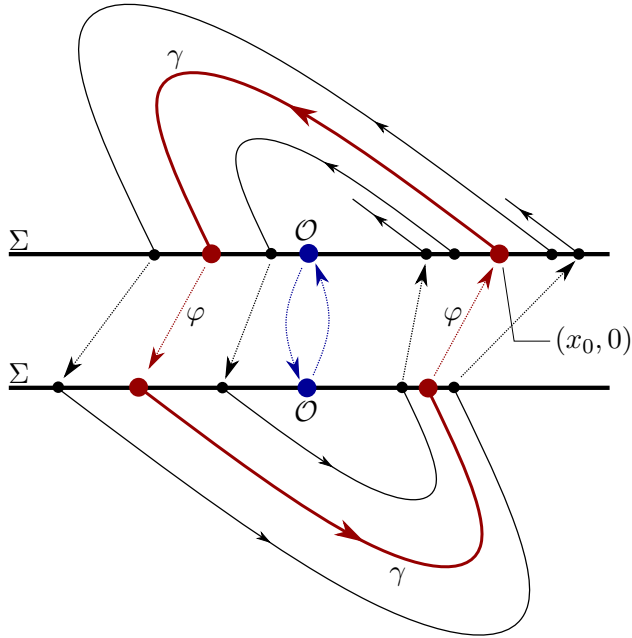
\end{example}

\section{Conclusion and further thoughts}\label{Sec6}

The main inspiration for this work is the sequence of papers~\cites{LFM1,LFM2,LFM3} in which is studied the crossing matching of different types of globally asymptotically stable vector fields. Such matching sometimes result in a globally asymptotically stable piecewise vector field, but not always.

By embedding these piecewise vector fields in a hybrid structure we were able to connect such results in a unique formulae that encapsulates the global dynamics of this new hybrid system. In particular, such formulae shows precisely when the origin is GAS and also indicates the bifurcation of a limit cycle that was previously impossible in the piecewise framework.

\section*{Acknowledgments}

The first author is partially supported by Funda\c c\~ao de Amparo \`a Pesquisa do Estado de Minas Gerais, Brazil, grant numbers APQ-02153-23, APQ-05207-23 and RED-00133-21.

The second author is partially supported by S\~ao Paulo Research Foundation (FAPESP), Brazil, grant 2021/01799-9. 

The second author is grateful for the hospitality of Universidade Federal de Itajub\'{a} (UNIFEI), where this study was developed.


\begin{thebibliography}{99}
	
\bibitem{Acc1}
\textsc{E. Accinelli, F. Martins, J. Oviedo, A.A. Pinto and L. Quintas}, 
\textit{Who controls the controller? A dynamical model of corruption},
J. Math. Sociol., \textbf{41} (2017), 220-247. 
\url{https://doi.org/10.1080/0022250x.2017.1388235}.

\bibitem{Acc2}
\textsc{E. Accinelli, F. Martins, A.A. Pinto, A. Afsar and B.M.P.M. Oliveira}, 
\textit{The power of voting and corruption cycles},
J. Math. Sociol., \textbf{46} (2022), 56-79. 
\url{https://doi.org/10.1080/0022250x.2020.1818077}.
	
\bibitem{ESS1}
\textsc{J. Bastos, C. Buzzi and P. Santana}, 
\textit{Evolutionary stable strategies and cubic vector fields},
NoDEA Nonlinear Differential Equations Appl., \textbf{31} (2024), Paper No. 13, 54 pp.  
\url{https://doi.org/10.1007/s00030-023-00894-4}.

\bibitem{ESS2}
\textsc{J. Bastos, C. Buzzi and P. Santana}, 
\textit{On structural stability of evolutionary stable strategies},
J. Differential Equations, \textbf{389} (2024), 190-227.  
\url{https://doi.org/10.1016/j.jde.2024.01.024}.

\bibitem{BKPS2023}
\textsc{I. Belykh, R. Kuske, M. Porfiri and D.J.W. Simpson}, 
\textit{Beyond the Bristol book: advances and perspectives in non-smooth dynamics and applications},
Chaos, \textbf{33} (2023), Paper No. 010402, 15 pp.  
\url{https://doi.org/10.1063/5.0138169}.

\bibitem{Bernardo2008}
\textsc{M. Bernardo, C.J. Budd, A.R. Champneys and P. Kowalczyk}, 
Piecewise-Smooth Dynamical Systems: Theory and Applications,
Applied Mathematical Sciences, vol 163. Springer, London, 2008.
\url{https://doi.org/10.1007/978-1-84628-708-4_2}.

\bibitem{LFM2}
\textsc{D.C. Braga, F.S. Dias, J. Llibre and L.F. Mello}, 
\textit{The discontinuous matching of two globally asymptotically stable crossing piecewise smooth systems in the plane do not produce in general a piecewise differential system globally asymptotically stable},
S\~ao Paulo J. Math. Sci., \textbf{17} (2023), 671–678. 
\url{https://doi.org/10.1007/s40863-023-00368-6}.

\bibitem{LFM3}
\textsc{D.C. Braga, F.S. Dias, J. Llibre and L.F. Mello}, 
\textit{The matching of two Markus-Yamabe piecewise smooth systems in the plane},
Nonlinear Anal. Real World Appl., \textbf{82} (2025), Paper No. 104254, 6 pp.  
\url{ https://doi.org/10.1016/j.nonrwa.2024.104254}.

\bibitem{LFM1}
\textsc{D.C. Braga, A.F. Fonseca and L.F. Mello}, 
\textit{The matching of two stable sewing linear systems in the plane},
J. Math. Anal. Appl., \textbf{433} (2016), 1142–1156.  
\url{https://doi.org/10.1016/j.jmaa.2015.08.037}.

\bibitem{Cim1997}
\textsc{A. Cima, A. van den Essen, A. Gasull, E. Hubbers and F. Ma\~nosas}, 
\textit{A polynomial counterexample to the Markus-Yamabe conjecture},
Adv. Math., \textbf{131} (1997), 453-457. 
\url{https://doi.org/10.1006/aima.1997.1673}.

\bibitem{CarNovTon2024}
\textsc{T. Carvalho, D.D. Novaes and D.J. Tonon}, 
\textit{Sliding mode on tangential sets of Filippov systems},
J. Nonlinear Sci., \textbf{34} (2024), Paper No. 70, 18 pp.  
\url{ https://doi.org/10.1007/s00332-024-10052-4}.

\bibitem{Duan}
\textsc{G.R. Duan and R.J. Patton}, 
\textit{A note on Hurwitz stability of matrices},
Automatica J. IFAC, \textbf{34} (1998), 509–511. 
\url{ https://doi.org/10.1016/s0005-1098(97)00217-3}.
	
\bibitem{Feb1995}
\textsc{R. Fe\textnormal{\ss}ler}, 
\textit{A proof of the two-dimensional Markus-Yamabe stability conjecture and a generalization},
Ann. Polon. Math., \textbf{62} (1995), 45–74.  
\url{https://doi.org/10.4064/ap-62-1-45-74}.
	
\bibitem{Frei1998}
\textsc{E. Freire, E. Ponce, F. Rodrigo and F. Torres}, 
\textit{Bifurcation sets of continuous piecewise linear systems with two zones},
Internat. J. Bifur. Chaos Appl. Sci. Engrg., \textbf{8} (1998), 2073–2097.  
\url{https://doi.org/10.1142/s0218127498001728}.

\bibitem{Fri1991} 
\textsc{D. Friedman}, 
\textit{Evolutionary games in economics}, 
Econometrica, \textbf{59} (1991), 637–666.  
\url{https://doi.org/10.2307/2938222}.

\bibitem{ESS3}
\textsc{A. Gasull, L.F.S. Gouveia and P. Santana}, 
\textit{On the limit cycles of a quartic model for evolutionary stable strategies},
Nonlinear Anal. Real World Appl., \textbf{84} (2025), Paper No. 104313, 17 pp.  
\url{https://doi.org/10.1016/j.nonrwa.2024.104313}.

\bibitem{Glu1995}
\textsc{A.A. Glutsyuk}, 
\textit{The asymptotic stability of the linearization of a vector field on the plane with a singular point implies global stability},
Funktsional. Anal. i Prilozhen., \textbf{29} (1995), 17–30, 95; translation in
Funct. Anal. Appl., \textbf{29} (1995), 238–247 (1996).  
\url{https://doi.org/10.1007/bf01077471}.

\bibitem{Guardia}
\textsc{M. Guardia, T.M. Seara and M.A. Teixeira}, 
\textit{Generic bifurcations of low codimension of planar Filippov systems},
J. Differential Equations \textbf{250} (2011), 1967–2023.
\url{https://doi.org/10.1016/j.jde.2010.11.016}.

\bibitem{Gut1995}
\textsc{C. Gutierrez}, 
\textit{A solution to the bidimensional global asymptotic stability conjecture},
Ann. Inst. H. Poincar\'e C Anal. Non Lin\'eaire, \textbf{12} (1995), 627–671.  
\url{ https://doi.org/10.1016/s0294-1449(16)30147-0}.

\bibitem{HSS} 
\textsc{H. Hofbauer, P. Schuster and K. Sigmund}, 
\textit{A note on evolutionary stable strategies and game dynamics}, 
J. Theoret. Biol., \textbf{81} (1979), 609–612. 
\url{ https://doi.org/10.1016/0022-5193(79)90058-4}.

\bibitem{Control} 
\textsc{Y. Iwatani and S. Hara}, 
\textit{Stability tests and stabilization for piecewise linear systems based on poles and zeros of subsystems}, 
Automatica J. IFAC, \textbf{42} (2006), 1685–1695.
\url{https://doi.org/10.1016/j.automatica.2006.06.009}.

\bibitem{LiLiu2023}
\textsc{Z. Li and X. Liu}, 
\textit{Impact limit cycles in the planar piecewise linear hybrid systems},
Commun. Nonlinear Sci. Numer. Simul., \textbf{119} (2023), Paper No. 107074, 16 pp.  
\url{ https://doi.org/10.1016/j.cnsns.2022.107074}.

\bibitem{LiSheZhaHao2016}
\textsc{S. Li, C. Shen, W. Zhang and Y. Hao}, 
\textit{The Melnikov method of heteroclinic orbits for a class of planar hybrid piecewise-smooth systems and application},
Nonlinear Dynam., \textbf{85} (2016), 1091–1104.  
\url{https://doi.org/10.1007/s11071-016-2746-9}.

\bibitem{LiWeiZhaKap2023}
\textsc{Y. Li, Z. Wei, W. Zhang and T. Kapitaniak}, 
\textit{Melnikov-type method for chaos in a class of hybrid piecewise-smooth systems with impact and noise excitation under unilateral rigid constraint},
Appl. Math. Model., \textbf{122} (2023), 506–523. 
\url{https://doi.org/10.1016/j.apm.2023.06.015}.

\bibitem{LliSan2025} 
\textsc{J. Llibre and P. Santana}, 
\textit{Limit cycles and chaos in planar hybrid systems}, 
Commun. Nonlinear Sci. Numer. Simul., \textbf{140} (2025), Paper No. 108382, 13 pp.  
\url{ https://doi.org/10.1016/j.cnsns.2024.108382}.

\bibitem{LliTei2018}
\textsc{J. Llibre and M.A. Teixeira}, 
\textit{Piecewise linear differential systems with only centers can create limit cycles?},
Nonlinear Dynam., \textbf{91} (2018), 249–255.  
\url{https://doi.org/10.1007/s11071-017-3866-6}.

\bibitem{Lou2024}
\textsc{X. Lou, Y. Li and R.G. Sanfelice}, 
\textit{Notions, stability, existence, and robustness of limit cycles in hybrid dynamical systems},
IEEE Trans. Automat. Control, \textbf{69} (2024), 4910–4925. 
\url{https://doi.org/10.1109/tac.2023.3340121}.

\bibitem{MarYam1960}
\textsc{L. Markus and H. Yamabe}, 
\textit{Global stability criteria for differential systems},
Osaka Math. J., \textbf{12} (1960), 305–317.  
\url{https://doi.org/10.18910/9397}.

\bibitem{Palm} 
\textsc{G. Palm}, 
\textit{A note on evolutionary stable strategies and game dynamics for $n$-person games}, 
J. Math. Biol., \textbf{19} (1984), 329–334. 
\url{https://doi.org/10.1007/bf00277103}.

\bibitem{Perko}
\textsc{L. Perko}, 
Differential Equations and Dynamical Systems. 
Third edition. Texts in Applied Mathematics, 7. Springer-Verlag, New York, 2001. 
\url{https://doi.org/10.1007/978-1-4613-0003-8}.

\bibitem{SanSer2025}
\textsc{P. Santana and L. Serantola}, 
\textit{On the stability of hybrid polycycles},
Qual. Theory Dyn. Syst., \textbf{24} (2025), Paper No. 82, 21 pp.  
\url{https://doi.org/10.1007/s12346-025-01242-w}.

\bibitem{SchSch2000}
\textsc{A. van der Schaft and H. Schumacher},
An Introduction to Hybrid Dynamical Systems.
Lecture Notes in Control and Information Sciences, 251. Springer-Verlag London, Ltd., London, 2000.
\url{https://doi.org/10.1007/BFb0109998}.

\bibitem{SchSig}
\textsc{P. Schuster and K. Sigmund}, 
\textit{Coyness, philandering and stable strategies},
Animal Behaviour, \textbf{29} (1981), 186-192. 
\url{https://doi.org/10.1016/s0003-3472(81)80165-0}.

\bibitem{WanZha2024}
\textsc{L. Wang and X. Zhang}, 
\textit{Invariant tori, topological horseshoes, and their coexistence in piecewise smooth hybrid systems},
Nonlinear Dyn., \textbf{112} (2024), 14617–14635. 
\url{https://doi.org/10.1007/s11071-024-09807-1}.

\end{thebibliography}
\end{document}